\documentclass[11pt]{article}

\usepackage{amssymb}
\usepackage{amsmath}
\usepackage{latexsym}
\usepackage{bm}
\usepackage{graphicx}
\usepackage{extarrows}
\usepackage{amsfonts}
\usepackage{amssymb}
\usepackage{indentfirst}
\usepackage{bbm}
\usepackage{xypic}
\usepackage{mathdots}
\usepackage{MnSymbol}
\usepackage[title]{appendix}

\usepackage{float}

\newtheorem{definition}{Definition}[section]
\newtheorem{theorem}{Theorem}[section]
\newtheorem{proposition}[theorem]{Proposition}
\newtheorem{lemma}[theorem]{Lemma}
\newtheorem{remark}[theorem]{Remark}
\newtheorem{corollary}[theorem]{Corollary}

\DeclareMathOperator{\Hom}{Hom}
\DeclareMathOperator{\spn}{span}
\DeclareMathOperator{\spt}{supp}
\DeclareMathOperator{\Aut}{Aut}
\DeclareMathOperator{\mil}{mil}
\DeclareMathOperator{\Ext}{Ext}
\DeclareMathOperator{\tr}{tr}
\DeclareMathOperator{\id}{id}

\begin{document}
\title{A Family of Projective Representations of the Thompson Group and Lifting Problems}

\author{Jun Yang}
\date{January 21, 2020}

\maketitle
\abstract{
The Thompson group F has a natural unitary representation on $H=L^2[0,1]$. 
With some projections, we construct a family of projective
unitary representations on a Fermionic Fock space associated with $H$. 
It comes from the representation of the associated CAR algebra. 
After $H^2(F;S^1)$ is obtained, we mainly study whether any of these projective unitary representations can be lifted to an ordinary one. 
We will discuss the lifting problem of these projective representations.
}
\tableofcontents

\section{Introduction}

In this paper, we mainly construct a family of projective unitary representations of the Thompson group $F$.
It is realized through the CAR algebra of $H=L^2([0,1])$ on a Fermionic Fock space.
Then, after getting $H^2(F;S^1)$, we will discuss the lifting problem.

In chapter 2, we review some basic facts of the representation theory of the CAR algebras over the Fermionic Fock spaces.
Given a projection $P$ on a Hilbert space $H$, the restricted unitary group is  $U_{res}(H)=\{u\in U(H)|[u,P] \text{~is Hilbert-Schmidt}\}$. 
By its action on the associated CAR algebra, 
there is an action on the Fermionic Fock spaces associated with the projection which gives us a projective representation. 
An interesting question is whether this projective representation can be lifted to an ordinary one when we focus on a subgroup. 
We mainly study two approaches to it: vacuum vectors and one parameter subgroups. 

In chapter 3, we give a short introduction to the Thompson group $F$ with some presentations. 
A unitary representation on $H=L^2[0,1]$ called Koopman representation is involved there.  
We also study the cohomology groups with coefficient in $s^1$ using the classifying space given by Quillen.

In chapter 4, we construct a projection and apply the theory of Section 2 to obtain a projective representation of the group $F$. 
We try to answer the question by giving a concrete computation to examine whether there is a lifting. 
We also generalize the special projection to the ones correspond to the elements in $SO(2)$ or $U(2)$. 

The author would like to thank his advisor Vaughan F.R. Jones for the work on construction of these representations. 
This work will never be done with out him. 
James Tenner also provided great help for the lifting problem. 

\section{A Projective Unitary Representation of the Restricted Unitary Group}
\subsection{The restricted unitary group}
Let $H$ be an infinite dimensional Hilbert space. Let $P\in B(H)$ be a nontrivial projection with infinite dimensional range and kernel.

\begin{definition}
$U_{res}(H)=\{u\in U(H)|[u,P] \text{~is Hilbert-Schmidt}\}$ is the restricted unitary group of $H$. 
\end{definition}

$U_{res}(H)$ is a Banach Lie group.  
We also define
\begin{center}
$\mathcal{A}=\{x\in B(H)|x=x^*, Px(1-P) \text{~is Hilbert-Schmidt}\}$
\end{center}
which will behave as the Lie algebra of $U_{res}(H)$. i.e. $e^{i\mathcal{A}}\subset U_{res}(H)$.

Take an element $X \in B(H)$. 
Let
\begin{equation*}
 X=\left(
  \begin{array}{cc}
  X^{(1,1)}&X^{(1,2)} \\
          X^{(2,1)}&X^{(2,2)}
 \end{array}
 \right), 
\end{equation*}
be the decomposition of $X$ with respect to the direct sum $H=PH\oplus (1-P)H$. 
Then we have 
\begin{enumerate}
    \item For any $u\in U(H)$, we have $u\in U_{res}(H)$ iff $u^{(1,2)},u^{(2,1)}$ are Hilbert-Schmidt. 
    \item For any self-adjoint $X\in B(H)$, we have $X \in \mathcal{A}$ iff $X^{(1,2)}$ is Hilbert-Schmidt. 
\end{enumerate}

We definte the index of a group element $u$ be the Fredholm index of $u^{(1,1)}=PuP$. 
\begin{center}
$i(u)=\text{index}(u^{(1,1)})$
\end{center}

\begin{theorem}[\cite{CHB}]
Suppose $u_0,u_1\in U_{res}(H)$. 
\begin{enumerate}
\item The connected components of $U_{res}(H)$ are
$U_{res}(H,k)=\{u\in U_{res}(H)|i(u)=k\}$. 
Moreover $i(u_0)=i(u_1)$ iff $|u_0-u_1|<2$ and $i(u_0)\neq i(u_1)$ iff $|u_0-u_1|=2$
\item If $i(u_0)=i(u_1)$, then there exists an $X \in \mathcal{A}$ such that
\begin{center}
$u_s=u_0e^{isX}$, $0\leq s\leq 1$
\end{center}
is a path in $U_{res}(H)$ connection $u_0,u_1$. And such $X$ is uniquely determined if its spectrum is contained in $(\pi,\pi)$. 
\end{enumerate}
\end{theorem}
We let $U_{res}^0(H)=U_{res}(H,0)$ denote the connected subgroup containing the identity.

\subsection{Fermionic Fock space representation of $U_{res}(H)$}

Starting with $H$ and its inner product $\langle\cdot,\cdot\rangle$,
there is a Fermionic Fock space $\wedge H$ which is also a Hilbert space spanned by
\begin{center}
$f_1\wedge\cdots\wedge f_k$, $k\in \mathbb{N}$ and $f_i\in H$ for $1\leq i \leq k$
\end{center}
with inner product
$\langle f_1\wedge\cdots\wedge f_k,g_1\wedge\cdots\wedge g_l\rangle=\delta_{g,l}\det(\langle f_i,g_j\rangle)$.

There is a CAR (canonical anticomutation relation) algebra $CAR(H)$ which is a complex algebra generated by the $\mathbb{C}-$linear symbols $\{a(f)|f\in H\}$ with the anticomutation relations
\begin{center}
$\{a(f),a(g)\}=0$ and $\{a(f),a(g)^*\}=\langle f,g\rangle$
\end{center}
where $\{X,Y\}=XY+YX$.

The algebra $CAR(H)$ has a representation $\pi$ on the Fermionic Fock space $\wedge H$ \cite{Tol} given by
\begin{enumerate}
\item $\pi(a(f))(g_1\wedge\cdots\wedge g_m)=f\wedge g_1\wedge\cdots\wedge g_m$,
\item $\pi(a(f)^{*})(g_1\wedge\cdots\wedge g_m)=\sum_{i=1}^{m}{(-1)}^{i-1}\langle g_i,f\rangle g_1\wedge\cdots\wedge\widehat{g_i}\wedge \cdots\wedge g_m  $
\end{enumerate}

\begin{lemma}[\cite{BSZ}]
The representation $\pi$ of $CAR(H)$ on $\wedge H$ is irreducible.
\end{lemma}

Let $P\in \mathcal{B}(H)$ be a projection.
There is a corresponding Fermionic Fock space defined by
\begin{center}
$\mathcal{F}=\wedge(PH)\widehat{\otimes}\wedge(P^{\bot} H)^{*}=\widehat{\otimes}_{m,n\geq 0}\mathcal{F}^{(m,n)}$.
\end{center}
where $\mathcal{F}^{(m,n)}=\wedge^m(PH)\widehat{\otimes}\wedge^n(P^{\bot} H)^{*}$. 
We have $\mathcal{F}$ is also a Hilbert space with the inner product from  $\wedge(PH)$ and $\wedge(P^{\bot} H)^{*}$.
If $\{p_i\}_{i\in \mathbb{N}}$ is an orthonormal basis of $PH$ and $\{q_i\}_{i\in \mathbb{N}}$ is an orthonormal basis of $P^{\perp}H$,
then we have a canonical orthonormal basis of $\mathcal{F}$.

\begin{lemma}
Given $\{p_i\}_{i\in \mathbb{N}}$ and $\{q_i\}_{i\in \mathbb{N}}$ as orthonormal basis of $PH$ and $P^{\perp}H$ respectively, then
\begin{center}
$\{p_{i_1}\wedge\cdots\wedge p_{i_k}\otimes q_{j_1}^*\wedge\cdots\wedge q_{j_l}^*|l,k\in \mathbb{N}\cup\{0\},1\leq i_1<\cdots<i_k,1\leq j_1<\cdots<j_l\}$
\end{center}
is an orthonormal basis of $\mathcal{F}$,
where $k,l=0$ stands for the vacuum vector $\Omega_1,\Omega_2$ in $\wedge(PH),\wedge(P^{\bot} H)^{*}$ respectively.
\end{lemma}
\begin{proof}
Note that $\epsilon_{i_1}\wedge\cdots\wedge \epsilon_{i_k}\otimes \eta_{j_1}^*\wedge\cdots\wedge \eta_{j_l}^*$ spans a dense subspace of $\mathcal{F}$ if $\epsilon_{i_r}\in PH$ and $\eta_{j_s}\in P^{\perp}H$.

Such a vector can be approximated by finite linear combinations of the vectors defined above.
One can easily check the these vectors are orthonormal, which completes the proof.
\end{proof}

We can define a representation $\pi_P$ of $CAR(H)$ on $\mathcal{F}$ given by
\begin{center}
$\pi_P(a(f))=a(Pf)\otimes 1 + 1\otimes a((P^{\bot}f)^{*})^{*}$.
\end{center}
where $a(\cdot)$ stands for the action $\pi$ (defined above) on the corresponding exterior space $\wedge(PH)$ or $\wedge((P^{\perp}H)^{*})$.
\begin{lemma}[\cite{BSZ}]
The representation $\pi_P$ of $CAR(H)$ on $\mathcal{F}$ is irreducible.
\end{lemma}

Given to projections $P,Q\in \mathcal{B}(H)$, we have the following result on the equivalence of representations $\pi_P,\pi_Q$.
\begin{theorem}[Segal's equivalence criterion\cite{BSZ}\cite{Wa}]
For two projections $P,Q\in \mathcal{B}(H)$, $\pi_P,\pi_Q$ are unitarily equivalent if and only if $P-Q$ is a Hilbert-Schmidt operator.
\end{theorem}

Given a unitary $u\in \mathcal{B}(H)$, the map $a(f)\mapsto a(uf)$ gives an automorphism of $CAR(H)$.
An interesting question is whether this $u$ can be realized by unitary elements in $\mathcal{B}(\mathcal{F})$.
\begin{definition}
$u$ is implemented in $\mathcal{F}$ if there is a unitary $U\in B(\mathcal{F})$ such that $\pi_P(a(uf))=U\pi_P(a(f))U^*$ for all $f\in H$.
\end{definition}

Then we have a criterion for the implementation of any given $u\in U(H)$.
\begin{corollary}
$u\in U(H)$ is implemented in $\mathcal{F}$ if $[u,P]$ is a Hilbert-Schmidt operator.
\end{corollary}
\begin{proof}
Let $Q=u^* P u$. Then we have $P-Q=P-u^* P u=u^*(uP-Pu)=u^*[u,P]$ is Hilbert-Schmidt.
Then, by the theorem above, there is a unitary $U$ such that $\pi_P(a(uf))=U\pi_Q(a(f))U^*=U\pi_{u^* P u}(a(f))U^*$ for all $f\in H$.

On the other hand, there is a unitary $V\in B(\mathcal{F}_Q,\mathcal{F})$ defined by
\begin{equation*}
\begin{aligned}
~~~~~~~~&V(Qg_1\wedge\cdots\wedge Qg_m\otimes (Q^\perp h_1)^*\wedge\cdots\wedge (Q^\perp h_n))\\
&= uQg_1\wedge\cdots\wedge uQg_m\otimes (uQ^\perp h_1)^*\wedge\cdots\wedge (uQ^\perp h_n)\\
&= uQu^*ug_1\wedge\cdots\wedge uQu^*ug_m\otimes (uQ^\perp u^*u h_1)^*\wedge\cdots\wedge (uQ^\perp u^*u h_n)\\
&=Pug_1\wedge\cdots\wedge Pug_m\otimes (P^\perp u h_1)^*\wedge\cdots\wedge (P^\perp u h_n)\in \mathcal{F}.
\end{aligned}
\end{equation*}
where $g_i,h_j\in H$ for $1\leq i\leq m, 1\leq j\leq n$.

Similarly, we have $V^*\in B(\mathcal{F},\mathcal{F}_Q)$ acts on $\mathcal{F}$ by
\begin{equation*}
\begin{aligned}
~~~~~~~~&V^*(Pg_1\wedge\cdots\wedge Pg_m\otimes (P^\perp h_1)^*\wedge\cdots\wedge (P^\perp h_n))\\
&= u^*Pg_1\wedge\cdots\wedge u^*Pg_m\otimes (u^*P^\perp h_1)^*\wedge\cdots\wedge (u^*P^\perp h_n)\in \mathcal{F}_Q.
\end{aligned}
\end{equation*}

It implies that $V^*\pi_P(a(uf))V=\pi_Q(a(f))$ for all $f\in H$.
Then, by the unitary equivalence of $\pi_P,\pi_Q$, $u$ is implemented as
\begin{center}
$\pi_P(a(uf))=V\pi_Q(a(f))V^*=(VU^*)\pi_P(a(f))(VU^*)^*$ for all $f\in H$
\end{center}
with $VU^*\in B(\mathcal{F})$ unitary.
\end{proof}

This gives us a projection unitary representation
\begin{center}
$\Gamma:U_{res}(H)\to PU(\mathcal{F})$. 
\end{center}
which can be shown to be irreducible. 

\begin{proposition}[\cite{HW}]
The representation $\Gamma$ is irreducible and $\Gamma(u)(\oplus_{m-n=k}\mathcal{F}^{(m,n)})\subset \oplus_{m-n=k+i(u)}\mathcal{F}^{(m,n)}$. 
\end{proposition}

\subsection{The vacuum vector}
Consider the  vacuum vector $\Omega=\Omega_1\otimes\Omega_2\in \mathcal{F}_P$, by Lemma 1.3 and action of $\pi_P(a(f))$, we have
\begin{center}
$\{\pi_P(a(p_{i_1})) \cdots \pi_P(a(p_{i_k})) \pi_P(a(q_{j_1})^*) \cdots \pi_P(a(q_{j_l})^*)\Omega |l,k\in \mathbb{N}\cup\{0\},1\leq i_1<\cdots<i_k,1\leq j_1<\cdots<j_l\}$
\end{center}
forms the orthonormal basis of $\mathcal{F}_P$.
That is to say $\overline{\pi_P(CAR(H))\Omega}^{\left\lVert \cdot \right\rVert}=\mathcal{F}_P$.

\begin{lemma}
The vacuum vector $\Omega=\Omega_1\otimes\Omega_2\in \mathcal{F}_P$ is the unique vector (up to a scalar) that annihilated by $\pi_P(a(p_i)^*),\pi_P(a(q_j))$ for all $i,j\in\mathbb{N}$.
\end{lemma}
\begin{proof}
It is easy to check $\Omega$ is annihilated by these operators.

For the uniqueness, let us consider the space $\Omega^{\perp}$ which is spanned (densely) by
\begin{center}
$\{v=\pi_P(a(p_{i_1})) \cdots \pi_P(a(p_{i_k})) \pi_P(a(q_{j_1})^*) \cdots \pi_P(a(q_{j_l})^*)\Omega |l,k\in \mathbb{N}\cup\{0\},lk\neq 0,1\leq i_1<\cdots<i_k,1\leq j_1<\cdots<j_l\}$
\end{center}
that at least one of $a(p_i)$ or $a(q_j)^*$ appears.
If there is such a vector $\Omega'$, there must be $\langle \Omega',v\rangle=0$ for any $v$ described above.
So $\Omega'\in \mathbb{C}\Omega$.
\end{proof}

Now, let us go back to the implementation of $\{u_g|g\in F\}$.
Suppose each $u_g$ is implemented by $U_g\in U(\mathcal{F}_P)$ such that
\begin{center}
$\pi_P(a(u_gf))U_g=U_g\pi_P(a(f))$ for all $f\in H$.
\end{center}
We have its action on $\Omega$ as
\begin{center}
$\pi_P(a(u_gf))U_g\Omega=U_g\pi_P(a(f))\Omega$ for all $f\in H$.
\end{center}
Now, replace $a(f)$ by $a(p_i)^*$ and $a(q_j)$ with $i,j\in\mathbb{N}$ and define
\begin{center}
$\Omega_g \overset{def}{=} U_g\Omega$
\end{center}

As $U_g$ is invertible, we have
\begin{corollary}
$\Omega_g$ is the unique vector (up to a scalar) that is annihilated by
\begin{center}
$a(u_gp_i)^*$ and $a(u_gq_j)$ with $i,j\in\mathbb{N}$,
\end{center}
\end{corollary}

This give us a first description of the action $U_g$.

Once we get the action of a $U_g$ on $\Omega$, i.e $\Omega_g= U_g\Omega$,
one may wonder the action of $U_g$ on the whole space $\mathcal{F}_P$.

Now, let us go back to the equation
\begin{center}
$\pi_P(a(u_gf))U_g\Omega=U_g\pi_P(a(f))\Omega$
\end{center}
for all $f\in H$.

Take an arbitrary vector from the basis mentioned above, i.e.
\begin{center}
$v=\pi_P(a(p_{i_1})) \cdots \pi_P(a(p_{i_k})) \pi_P(a(q_{j_1})^*) \cdots \pi_P(a(q_{j_l})^*)\Omega $
\end{center}
with $l,k\in \mathbb{N}\cup\{0\},1\leq i_1<\cdots<i_k,1\leq j_1<\cdots<j_l$.
Let $U_g$ act on this vector, we have
\begin{equation*}
\begin{aligned}
U_gv&=U_g\pi_P(a(p_{i_1})) \cdots \pi_P(a(p_{i_k})) \pi_P(a(q_{j_1})^*) \cdots \pi_P(a(q_{j_l})^*)\Omega\\
&=\pi_P(a(u_gp_{i_1}))U_g \pi_P(a(p_{i_2}))\cdots \pi_P(a(p_{i_k})) \pi_P(a(q_{j_1})^*) \cdots \pi_P(a(q_{j_l})^*)\Omega\\
&=\cdots=\pi_P(u_ga(p_{i_1})) \cdots \pi_P(a(u_gp_{i_k})) \pi_P(a(u_gq_{j_1})^*) \cdots \pi_P(a(u_gq_{j_l})^*)U_g\Omega\\
&=\pi_P(u_ga(p_{i_1})) \cdots \pi_P(a(u_gp_{i_k})) \pi_P(a(u_gq_{j_1})^*) \cdots \pi_P(a(u_gq_{j_l})^*)\Omega_g
\end{aligned}
\end{equation*}
which is to say
\begin{corollary}
Once $\Omega_g$ is known, the action of $U_g$ on $\mathcal{F}_P$ is explicit by its action of the orthonormal basis given above.
\end{corollary}

\subsection{Implementation of one parameter subgroups}

\begin{theorem}[\cite{L76},Theorem 1, Lemma 3.9]
There exists a map $W:X\mapsto W(sX) $ from $\mathcal{A}$ to $\pi_P(CAR(H))''\subset B(\mathcal{F})$  such that $W(sX)$
is a strongly continuous unitary one parameter subgroup fulfilling
\begin{center}
$\pi_P(a(e^{isX}f))=W(sX)\pi_P(a(f))W(sX)^{-1}$, $f\in H$. 
\end{center}
And the vacumm vector $\Omega\in D((d/ds)W(sX)_{s=0})$. 

Moreover, for $X,Y\in\mathcal{A}$, we have
\begin{center}
$W(tX)W(sY)W(tX)^{-1}=W(e^{itX}sYe^{-itX})e^{ib(tX,sY)}$
\end{center}
where $b(tX,sY)=-2 \operatorname{Im} \int_{0}^{t}\tr(PX(1-P)Ye^{irX}sYe^{-irX})$. 
\end{theorem}

If $X,Y$ above commute, i.e. $[X,Y]=0$, there is a simple formula for the commutator. 
\begin{corollary}
Assume $X,Y\in\mathcal{A}$ with $[X,Y]=0$, then
\begin{center}
$W(X)W(Y)W(X)^{-1}=W(Y)e^{-2i\operatorname{Im}\tr(PX(1-P)Y)}$
\end{center}
\end{corollary}
\begin{proof}
As $[X,Y]=0$, $e^{itX}sYe^{-itX}=sY$. 
Then let $r=s=1$
\end{proof}

This gives the lifting of some commutators of commuting elements in $U_{res}^0(H)$. Here we let $\Gamma$ denote the lifting in $U(\mathcal{F})$.

\begin{corollary}
Assume $u_0,u_1\in U_{res}^0(H)$ and $X_0,X_1\in \mathcal{A}$ such that $u_0=e^{iX_0},u_0=e^{iX_1}$. 
If $[X,Y]=0$, then $[u_0,u_1]=1$ and 
\begin{center}
$[\Gamma(u_0),\Gamma(u_1)]=e^{-2i\operatorname{Im}\tr(PX_0(1-P)X_1)}$
\end{center}
\end{corollary}

\section{The Thompson group $F$}

The Thompson group $F$ is a finitely generated group presented by
\begin{center}
$\langle A,B|[AB^{-1},A^{-1}BA]=[AB^{-1},A^{-2}BA^2]=\text{id} \rangle$
\end{center}
And there is also a presentation given by
\begin{center}
$\langle x_0,x_1,\dots|x_i^{-1}x_nx_i=x_{n+1}\text{~for~all~}i<n \rangle$.
\end{center}

We review some basic facts about $F$ that we will need later. Most of these can be found in introductive materials of the Thompson group \cite{Be}\cite{CFP}.

\subsection{Dyadic automorphism presentations of $F$}

\begin{definition}
Let $I,J$ be two real intervals. A homeomorphism $f:I\to J$ is called a dyadic piecewise linear homeomorphism if it satisfies:
\begin{enumerate}
	\item $f$ is piecewise linear with finite many singular points,
	\item the coordinates of each singular points is dyadic rational, i.e. $n/2^m$ for $n,m\in \mathbb{N}$,
	\item each slope of $f$ is an integral power of $2$
\end{enumerate}
And we let $\Aut_{\text{DPL}}(I)$ denote all the dyadic piecewise linear homeomorphisms from the interval $I$ to itself.
\end{definition}

It is well-known that $F=\Aut_{\text{DPL}}([0,1])$.
In this way, the generators $A,B$ can be presented as

\begin{equation*}
A(x)=
\left\{
             \begin{array}{ll}
             x/2, & \text{if~} x\in [0,1/2) \\
             x-1/2, & \text{if~} x\in [1/2,3/4)  \\
             2x-1, & \text{if~} x\in [3/4,1]
             \end{array}
\right.
\end{equation*}

\begin{equation*}
B(x)=
\left\{
             \begin{array}{ll}
             x, & \text{if~} x\in [0,1/2) \\
             x/2+1/4, & \text{if~} x\in [1/2,3/4) \\
             x-1/2, & \text{if~} x\in [3/4,7/8)  \\
             2x-1, & \text{if~} x\in [7/8,1]
             \end{array}
\right.
\end{equation*}
One can check such two maps satisfy
\begin{center}
$[AB^{-1},A^{-1}BA]=[AB^{-1},A^{-2}BA^2]=\text{id}$
\end{center}
and generate the group $F$.

For each $g\in F$, it can be written as a product of powers of $A,B$ and then $g$ also gives an element in $\Aut_{\text{DPL}}(I)$.
\begin{figure}
  \centering
  \includegraphics[width=15cm]{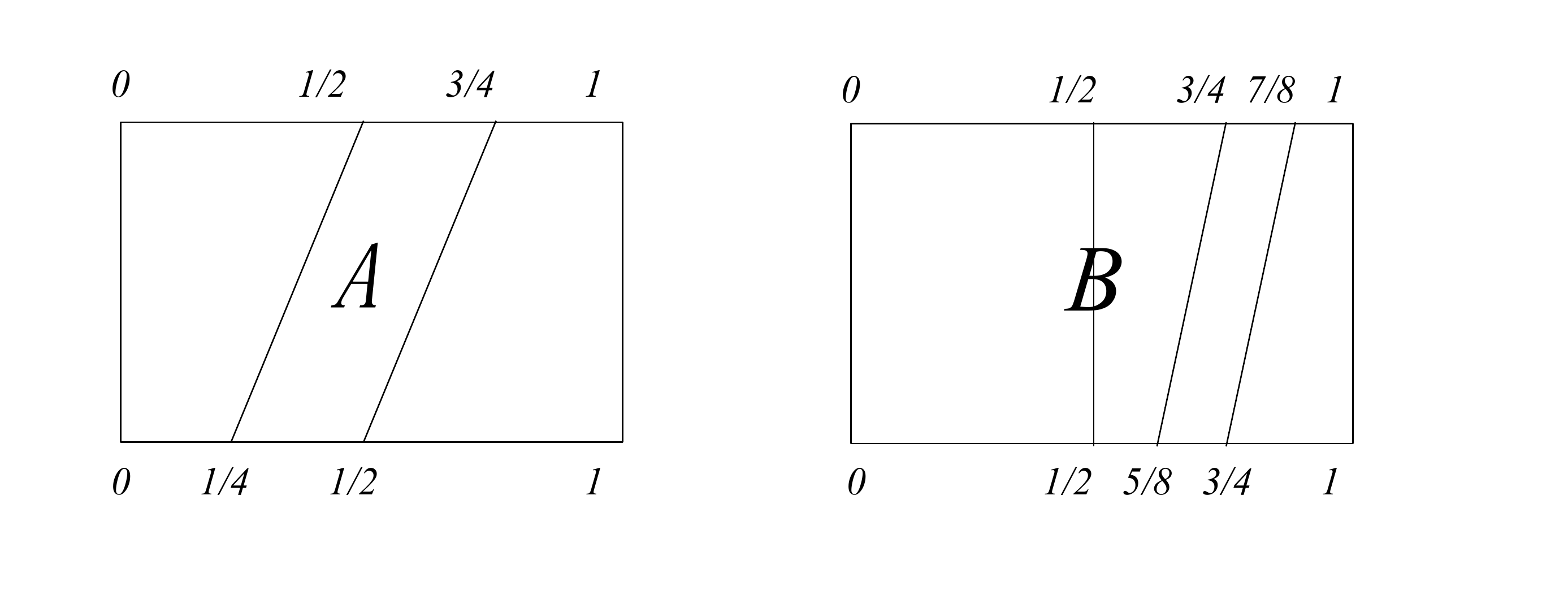}\\
  \caption{The two generators $A,B$}
\end{figure}
This leads to the description of {\it dyadic subdivision} of $[0,1]$ by repeating insertion of midpoint.
For example, the canonical map from one dyadic subdivision to another of equal number of subintervals is an element of $\Aut_{\text{DPL}}([0,1])$.

\begin{definition}
Given any $g\in F$, we define the minimal interval length, denoted by $\mil(g)$, to be the minimal length of interval that contains no singular points.

We also define $n_g=-\log_2 (\mil(g))$ and call this the level of $g$.
\end{definition}
One can easily check $\mil(g)$ is always an integral power of $2$ and hence $n_g\in \mathbb{N}$.
For example, we have $\mil(A)=1/4$, $\mil(B)=1/8$ and $n_A=2,n_B=3$.

\begin{lemma}
For any $g \in F$ with a reduced word form $g=A^{\alpha_1}B^{\beta_1}\cdots A^{\alpha_s}B^{\beta_s}$ ($\alpha_i,\beta_j\in \mathbb{Z}\backslash \{0\}$ with $\alpha_1,\beta_s=0$ allowed), we have
\begin{enumerate}
\item $n_g\leq \sum_{1\leq k \leq s}(2|\alpha_k|+3|\beta_k|)$;
\item $\mil(g)\geq \frac{1}{\prod_{1\leq k\leq s} 4^{|\alpha_k|}8^{|\beta_k|}}$
\end{enumerate}
\end{lemma}

The proof is straightforward by induction of $s$ and $\sum_{1\leq k \leq s}(|\alpha_k|+|\beta_k|)$ using the range of slopes.

There is another presentation of $F$ by binary trees \cite{Be}.
Given a dyadic subdivision of the interval $[0,1]$, there is a obvious binary tree corresponding to it.
And one can shown such a correspondence is one-to-one.
In this way, the group $F$ will acts on the set of binary trees.
Such notations will only be used in chapter 4.1.
\begin{figure}
  \centering
  \includegraphics[width=10cm]{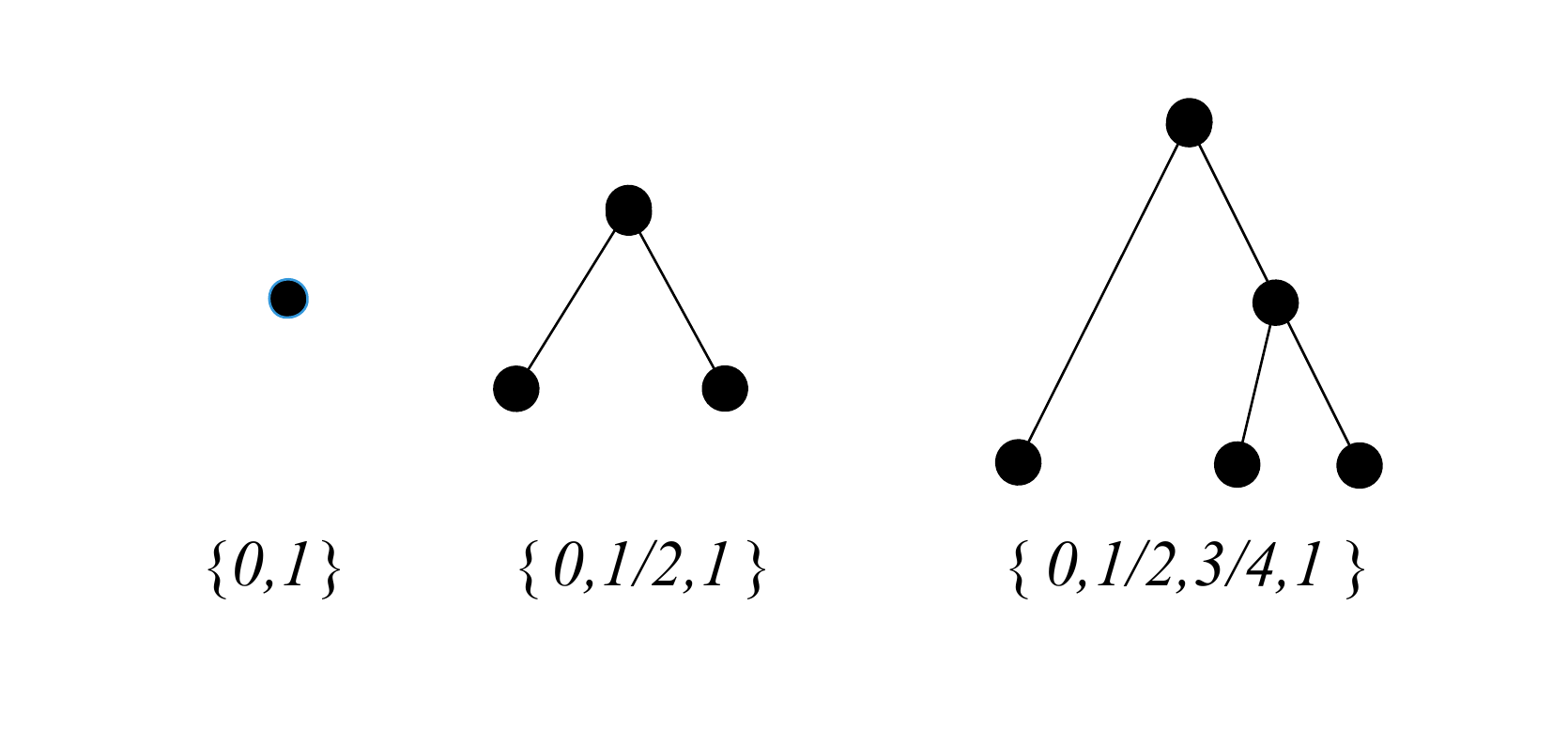}\\
  \caption{Dyadic subdivisions and binary trees correspondence}\label{1}
\end{figure}

\subsection{Koopman representation on $L^2([0,1])$}

There is a canonical Koopman representation $u$ of $F$ on the Hilbert space $H=L^2([0,1])$ which is defined by
\begin{center}
$(u(g)f)(x)=f(g^{-1}x)\sqrt{\frac{dg_{*}\mu}{d\mu}}(x)$, $f\in H$ and $x\in [0,1]$
\end{center}
where $\mu$ is the usual measure of $[0,1]$, $g_{*}\mu$ is defined by $g_{*}\mu(A)=\mu(g^A)$ and the quotient is the Radon-Nikodym Derivative.

Artem Dudko \cite{Du} proved this representation is irreducible by introducing the {\it measure contracting} action
\begin{definition}
A group $G$ acts on a probability space $(X,\mu)$ that is measure class preserving. It is called measure contracting if  for any measurable subset $A\subset X$, any $M,\varepsilon \in \mathbb{R}$, there is $g\in G$ such that
\begin{enumerate}
\item $\mu(\spt(g)\backslash A)<\varepsilon$;
\item $\mu(\{x\in A|\sqrt{\frac{d\mu(g(x))}{d\mu(x)}}<M^{-1}\})>\mu(A)-\varepsilon$.
\end{enumerate}
\end{definition}

The main theorem connecting this definition with irreducibility is
\begin{theorem}[\cite{Du}]
For any ergodic measure contraction action of a group $G$ on a probability space $(X,\mu)$, the associated Koopman representation of $G$ is irreducible.
\end{theorem}

Then, by checking the Koopman representation $u:F\to U(H)$ is ergodic and measure contracting, $u$ is irreducible as a corollary.

\subsection{Cohomology groups of $F$}

We will use classifying space to get the homology and cohomology groups of $F$.
Given a group $G$, there is a classifying space $BG$ whose homology and cohomology groups are the same as $G$ \cite{Ro}.
But we will mainly discuss the classifying space of a given small category defined by Quillen (see \cite{Qui} or \cite{Sri}).

Let $\mathcal{C}$ be a small category. The nerve of $\mathcal{C}$, denoted $N\mathcal{C}$, is defined to be the following simplicial set:
\begin{center}
$A_1\stackrel{f_1}{\longrightarrow}A_2\stackrel{f_1}{\longrightarrow}\cdots  A_{n-1}\stackrel{f_{n-1}}{\longrightarrow}A_n$,
\end{center}
where the $A_i$'s are objects in $\mathcal{C}$ and each $f_i:A_i\to A_{i+1}$ is a morphism.
Then the geometric realization of $N\mathcal{C}$ is defined to be $B\mathcal{C}$.

The most basic example is the category $\mathcal{C}_0$ of a partially ordering set of two element $\{0,1\}$ with $0\leq 1$.
Then there is only one simplex $0\to 1$ and hence we have $B\mathcal{C}_0=[0,1]$.

Kenneth Brown \cite{Br} gave the structure of cohomology ring $H^*(F;\mathbb{Z})$ by the homology of classifying spaces of several categories.
Here, as a review, we will give a quick and direct outline of the proof with \cite{Qui}.

Firstly, there are three categories related to $F$
\begin{enumerate}
\item The category $\mathcal{F}$

Let $\mathcal{F}$ be a (small) category whose objects are intervals $[0,n]$ ($n\geq 1$) and morphisms are the dyadic piecewise linear homeomorphisms.
Let $|\mathcal{F}|$ be its geometric realization for small categories defined by Quillen \cite{Qui}.
It is obvious the automorphism group of one object is exactly $F$.
Then, by \cite{Qui} again, $|\mathcal{F}|$ is an Eilenberg-MacLane complex of type $K(F,1)$.

\item The category $\mathcal{S}$

Let $\mathcal{S}$ be a (small) category whose objects are the same as $\mathcal{F}$.
But its morphisms are restricted to be {\it subdivision maps}: For any two intervals $[0,n+k],[0,n]$ ($k\leq 1-n$), take dyadic subdivision (section 2.2) of $[0,n+k]$ into $[0,1],\dots,[n+k-1,n+k]$ and $ [0,n]$ into any dyadic subdivision with $n+k$ parts, then the dyadic piecewise linear homeomorphisms is defined by correspondence of two pairs of $n+k+1$ points in $[0,n+k],[0,n]$ respectively.
$|\mathcal{S}|$ is an Eilenberg-MacLane complex of type $K(F,1)$.

\item The category $\mathcal{B}$

The category $\mathcal{B}$ is a POSET whose objects are any binary trees.
And for two objects $B,C$, $B\leq C$ if $C$ can be extended from $B$ by the usual binary tree expansion.
There is an obvious action of $F$ on $\mathcal{B}$.
And the geometric realization $|\mathcal{B}|$ is contractible with a free $F$ action and the quotient is just $|\mathcal{S}|$.
\end{enumerate}

With these three categories and geometric realizations above, we will introduce a new complex $X$, from which we can compute the homology and cohomology groups of $F$.

Let $L$ be a binary forest.
An {\it elementary expansion} of $L$ is expansion of $L$ of some different nodes belong to $L$, but not any expansion at a new node.
This implies the hight will plus at most one after this expansion.
We will write $L\preceq M$ if $M$ is an elementary expansion of $L$.
Note that $L=L_0 \preceq L_1 \preceq \cdots \preceq L_p=M $ will not imply $L\preceq M$.
But $L=L_0 \leq L_1 \leq \cdots \leq L_p=M$ with $L\preceq M$ will imply $L=L_0 \preceq L_1 \preceq \cdots \preceq L_p=M$.
By \cite{Qui} again, the elementary expansions form the elementary simplices $\tilde{X}$ which is an $F$-invariant subcomplex of $|\mathcal{B}|$.

Then, by passing to the quotient of action by $F$, we get a subcomplex $X\subset |\mathcal{S}|$.
It has one cell for each chain $L=L_0 \leq L_1 \leq \cdots \leq L_p=M$ with $L\preceq M$.
And by \cite{Ste}, we can decompose $X$ into cubes.

Recall that the geometric realization of the POSET $\{0,1\}$ is $[0,1]$.
Given any $L\preceq M$ where $M$ is a $k$-fold elementary expansion of $L$, let $[L,M]$ denotes all chains of elementary expansions from $L$ to $M$.
We have the following interesting result:
\begin{center}
the geometric realization $|[L,M]|=[0,1]^k$ as $[L,M]\simeq \{0,1\}^k$.
\end{center}
And the relative interior of this $k$-cube is the union of the open simplices corresponding to the chain $L=L_0 \leq L_1 \leq \cdots \leq L_p=M$ with $L\preceq M$.

Moreover, there is a natural product $\mathcal{F}\times \mathcal{F}\to \mathcal{F}$ by just gluing the objects and connect the two morphisms.
This makes $\mathcal{F}$ a semigroup with $|\mathcal{S}|,X$ as subsemigroups.
And as \cite{Br}, one can check $X$ is finite generated as semigroups by two elements $v,e$ (Figure 3):

\begin{figure}
  \centering
  \includegraphics[width=10cm]{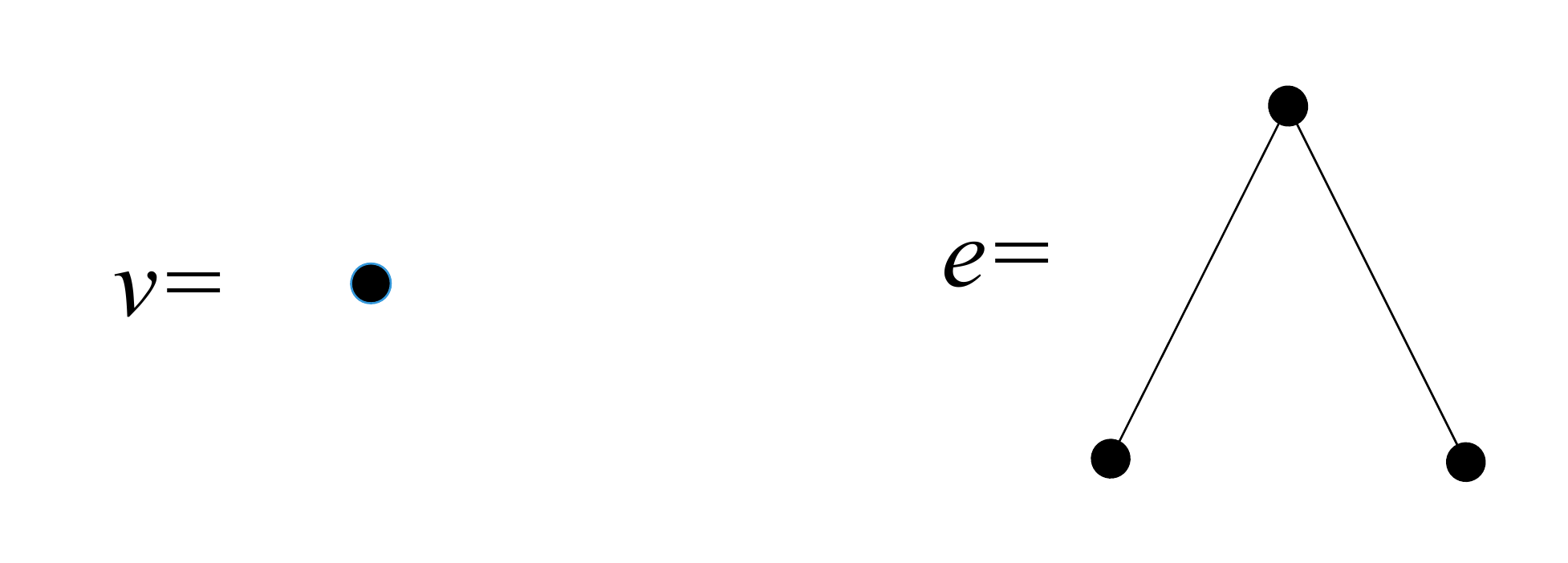}\\
  \caption{The two generators $v,e$}
\end{figure}

Let $C=C(X)$ be the cellular chain complex of $X$.
Then $C$ is a differential graded ring without identity.
One can check $\deg(v)=0,\deg(e)=1$ and the differential rule:
\begin{center}
$\partial(e)=v^2-v$ and $\partial(xy)=\partial(x)\cdot y+(-1)^{\deg(x)} x\cdot \partial(y)$.
\end{center}

Then, follows \cite{Br}, we have:
\begin{proposition}
There is are ring isomorphisms
\begin{enumerate}
\item $H_*(F;\mathbb{Z})\cong \mathbb{Z}[\varepsilon,\zeta]$ with relations  $\varepsilon^2=\varepsilon,\varepsilon\zeta=\zeta-\zeta\varepsilon$.
\item $H^*(F;\mathbb{Z})\cong\Gamma(u)\bigotimes \bigwedge (a,b)$ with $\deg{a}=\deg{b}=1$ and $\deg{u}=2$, where $\Gamma(u)$ is the ring generated by $u^{(i)}=u^i/i!$ over $\mathbb{Z}$.
\end{enumerate}
\end{proposition}

So there is $H^2(F;\mathbb{Z})=\mathbb{Z}\bigoplus\mathbb{Z}$ from the degree$=2$ terms.
Then, by the universal coefficient theorem \cite{Wei}
\begin{center}
$0\rightarrow \Ext_R^1(H_{n-1}(P),M)\rightarrow H^n(\Hom_R(P,M))\rightarrow \Hom_R(H_n(P),M)\rightarrow 0$
\end{center}
with $R=\mathbb{Z},M=S^1,n=2$ and $\Ext_\mathbb{Z}^1(H_{1}(F),S^1)=0,\Hom_\mathbb{Z}(H_2(F),S^1)=S^1\times S^1$ from the proposition above.
Then we have
\begin{corollary}
$H^2(F;S^1)=S^1\times S^1$.
\end{corollary}

That is to say not every projective unitary representation of $F$ can be lifted to a unitary one. We will check in details whether the representation $\Gamma$ can be lifted.

\section{A Projective Representation of $F$}

\subsection{Construction of the projective representation }

Consider a basis of $H=L^2([0,1])$ with dyadic support defined by
\begin{equation*}
\begin{aligned}
&f_{0,0}=1_{[0,1]}\\
&f_{1,0}=1_{[0,1/2)}-1_{[1/2,1]}\\
&f_{2,0}=\sqrt{2}\cdot 1_{[0,1/4)}-\sqrt{2}\cdot 1_{[1/4,1/2)},~~f_{2,1}=\sqrt{2}\cdot 1_{[1/2,3/4)}-\sqrt{2}\cdot 1_{[3/4,1]}
\end{aligned}
\end{equation*}
and $f_{n,k}=\sqrt{2}^{n-1}\cdot 1_{[2k/2^n,2k+1/2^n)}-\sqrt{2}^{n-1}\cdot 1_{[2k+1/2^n,2k+2/2^n)}$ is defined for all $n\in\mathbb{N},0\leq k \leq 2^{n-1}-1$.

One can check this forms an orthonormal basis of $H$.
It can also be renumbered lexicographically as $\{f_i\}_{i\in \mathbb{N}}$.
We will construct another othonormal basis $\{p_{n,t},q_{n,t}\}_{n\in \mathbb{N},0\leq t\leq 2^{n-2}-1}$ from $\{f_{n,k}\}$.
\begin{equation*}
\begin{aligned}
p_{1,0}&=\frac{\sqrt{2}}{2}(f_{0,0}+f_{1,0}),&q_{1,0}&=\frac{\sqrt{2}}{2}(f_{0,0}-f_{1,0})\\
p_{2,0}&=\frac{\sqrt{2}}{2}(f_{2,0}+f_{2,1}),&q_{2,0}&=\frac{\sqrt{2}}{2}(f_{2,0}-f_{2,1})\\
p_{n,t}&=\frac{\sqrt{2}}{2}(f_{n,2t}+f_{n,2t+1}),&q_{n,t}&=\frac{\sqrt{2}}{2}(f_{n,2t}-f_{n,2t+1})
\end{aligned}
\end{equation*}
and we also renumber it lexicographically as $\{p_i,q_i\}_{i\in \mathbb{N}}$.

Let $K_n$ be the subspace of $H$ spanned by $\{p_{n,t}\}_{0\leq t\leq 2^{n-2}-1}$.
Let $K$ be the subspace of $H$ spanned by $\{p_i\}_{i\in \mathbb{N}}$ and $P\in \mathcal{B}(H)$ is the projection onto $K$.
We have that $\ker{P}=\spn\{q_i\}_{i\in \mathbb{N}}$.
The question is that whether the Koopman representation of $u:F\to \mathcal{U}(H)$ (via $g\mapsto u_g$) can be implemented in $\mathcal{F}_P$.

\begin{lemma}
For any $g\in F$, there is a $n_g\in \mathbb{N}$ such that $u_g(p_{n,t})\in K$ and $u_g(q_{n,t})\in K^{\perp}$ for all $n\geq n_g$.
\end{lemma}
\begin{proof}
Given $g\in F$, suppose $g$ can be presented by a reduced form of $A^{\alpha_1}B^{\beta_1}\cdots A^{\alpha_s}B^{\beta_s}$.
By lemma 2.1, We have that the minimal interval length $mil(g)\geq \frac{1}{\prod_{1\leq k\leq s} 4^{|\alpha_k|}8^{|\beta_k|}}$.
Take $N_g$ to be the $2+\min\{n\in\mathbb{N}|\frac{1}{n}<mil(g)\}$.

Then, for $p_{n,t}\in K_n$ with $n\geq N_g$, we have all $\spt{p_{n,t}}$ are the dyadic intervals of $g$ without singularities.
This implies
\begin{center}
$u_g(p_{n,t})\in \bigoplus _{i=-n_g}^{n_g}K_{n+i}\subset K$.
\end{center}
and the proof is similar for $q_{n,t}$.
\end{proof}

\begin{lemma}
For any $g\in F$, $[u_g,P]$ is a Hilbert-Schmidt operator.
\end{lemma}
\begin{proof}
By the lemma above, there is $n_g\in\mathbb{N}$ such that $u_g(p_{n,t})\in K$ and $u_g(q_{n,t})\in K^\perp$ for all $n\geq n_g$.
After renumbering, there is $N_g\in \mathbb{N}$ such that $u_g(p_k)\in K$ and $u_g(q_k)\in K^\perp$ for all $n\geq N_g$.
Then we can get the Hilbert-Schmidt norm of $[u_g,P]$ is bounded by the following inequality.

\begin{equation*}
\begin{aligned}
\left\lVert [u_g,P]\right\rVert_2^2&=\sum\limits_{k=1}^{\infty}\left\lVert [u_g,P]p_k\right\rVert+\sum\limits_{k=1}^{\infty}\left\lVert [u_g,P]q_k\right\rVert\\
&=\sum\limits_{k=1}^{\infty}\left\lVert (u_g-Pu_g)p_k\right\rVert+\sum\limits_{k=1}^{\infty}\left\lVert Pu_g q_k\right\rVert\\
&=\sum\limits_{k=1}^{N_g}(\left\lVert (u_g-Pu_g)p_k\right\rVert+\left\lVert Pu_g q_k\right\rVert)+\sum\limits_{k=N_g+1}^{\infty}(\left\lVert (u_g-Pu_g)p_k\right\rVert+\left\lVert Pu_g q_k\right\rVert)\\
&=\sum\limits_{k=1}^{N_g}(\left\lVert (u_g-Pu_g)p_k\right\rVert+\left\lVert Pu_g q_k\right\rVert)<\infty.
\end{aligned}
\end{equation*}
So $[u_g,P]$ is a Hilbert-Schmidt operator.
\end{proof}

Then, by the corollary 3.5, we can directly get:
\begin{corollary}
The Thompson group $F$ is implemented in $\mathcal{F}_P$ through its Koopman representation $\{u_g|g\in F\}\in \mathcal{U}(H)$.
\end{corollary}

Up to now, for any $g\in F$, there is a unitary $U_g$ such that
\begin{center}
$\pi_P(a(u_gf))=U_g\pi_P(a(f))U_g^*$ for all $f\in H$.
\end{center}
As $\pi_P$ is irreducible, such a $U_g$ is unique up to a phase when it exists.
So, by passing to the projective unitary group $PU(\mathcal{F}_P)=U(\mathcal{F}_P)/(S^1\cdot I)$, we have a projective unitary representation
\begin{center}
$\Gamma:F\to PU(\mathcal{F}_P)$
\end{center}
as the composition of
\begin{center}
$\Gamma:F\stackrel{u}{\longrightarrow} \mathcal{U}(H)\stackrel{U}{\longrightarrow} U(\mathcal{F}_P)\stackrel{\cdot/(S^1\cdot I)}{\longrightarrow} PU(\mathcal{F}_P)$
\end{center}
by $g\mapsto u_g\mapsto U_g\mapsto \Gamma_g\in PU(\mathcal{F}_P)$.

Remark: The map $U$ is just one between sets, not a group homomorphism.

\subsection{Criternion for the lifting}
Now, we suppose there is a lifting $U:F\to U(\mathcal{F}_P)$ of $\Gamma:F\to PU(\mathcal{F}_P)$.
That is to say $\pi(U(g))=\Gamma(g)$ for all $g\in F$, where $\pi:U(\mathcal{F}_P)\to PU(\mathcal{F}_P)$ is the quotient map.

Since $F$ is finitely generated, it is enough to consider the lifting of $\Gamma_A,\Gamma_B$.
We can choose two arbitrary $R,S\in U(\mathcal{F}_P)$ such that $\pi(R)=\Gamma_A,\pi(S)=\Gamma_B$.
Then there must be two complex numbers $c_A,c_B\in S^1$ such that
\begin{center}
$U_A=c_A\cdot R$, $U_B=c_A\cdot S$.
\end{center}
And we will also have the lifting of $\Gamma_{A^*},\Gamma_{B^*}$ are $U_{A^{-1}}=c_A^{-1}\cdot R^{-1}$, $U_{B^{-1}}=c_B^{-1}\cdot S^{-1}$ respectively.
Then we have:
\begin{proposition}
With the lifting of $\Gamma_A,\Gamma_B$ (hence also $\Gamma_{A^{-1}},\Gamma_{B^{-1}}$) fixed as $U_A,U_B$,
there is a lifting $U:F\to U(\mathcal{F}_P)$ of $\Gamma:F\to PU(\mathcal{F}_P)$ if and only if the following conditions are satisfied:
\begin{enumerate}
\item $[U_AU_{B^{-1}},U_{A^{-1}}U_BU_A]=1$;
\item $[U_AU_{B^{-1}},U_{A^{-1}}^2U_BU_A^2]=1$.
\end{enumerate}

Moreover, whether it is satisfied is independent on the choice of $c_A,c_B$
\end{proposition}
\begin{proof}
Given $U_A,U_B$, $\Gamma$ can be lifted to a unitary $U$ is equivalent whether $U:F\to U(\mathcal{F}_P)$ is a homomorphism.
Obviously, this is equivalent to the two conditions above.

Moreover, as the two relations are both commutators.
The commutators are always $[RS^{-1},R^{-1}SR],[RS^{-1},R^{-2}SR^2]$, which are independent with the choice of $c_A,c_B$.
\end{proof}

\subsection{The infinitesimal generators of $F$}

We let $C=AB^{-1},D=A^{-1}BA^1$ and $E=A^{-2}BA^2$ so the relations of $F$ are
\begin{center}
$[C,D]=1,[C,E]=1$. 
\end{center}

\begin{lemma}
The Koopman representation of $F$ is in $U_{res}^0(H)$ with the projection $P$ given in section 3.1. 
\end{lemma}
\begin{proof}
It suffices to show that $i(u_A)=i(u_B)=0$ which is proved in Appendix. Then it follows by $i(gh)=i(g)i(h)$ for $g,h\in U_{res}(H)$.  
\end{proof}

\begin{theorem}
There exist infinitesimal generators $X_C,X_D,X_E \in \mathcal{A}$ with 
$u_C=e^{iX_C},u_D=e^{iX_D},u_E=e^{iX_E}$ and $e^{itX_C},e^{itX_D},e^{itX_E}\in U_{res}^0(H)$ for $0\leq t\leq 1$ such that
\begin{center}
$[X_C,X_D]=[X_C,X_E]=0$. 
\end{center}
\end{theorem}
\begin{proof}
The claim $e^{itX_C},e^{itX_D},e^{itX_E}\in U_{res}^0(H)$ is true by Theorem 1.1. 

Now we let $H_C=L^2[0,\frac{3}{4}]$ and $H_D=L^2[\frac{3}{4},1]$. Then $H=H_C\oplus H_D$. By the following graphs presentation of $u_C,u_D\in U(H)$ actiong on $[0,1]$, there is
\begin{center}
$u_C\in B(H_C)\oplus \id_{H_D}$ and $u_D\in \id_{H_C}\oplus B(H_D)$.
\end{center}

\begin{figure}[H]
  \centering
  \includegraphics[width=10cm]{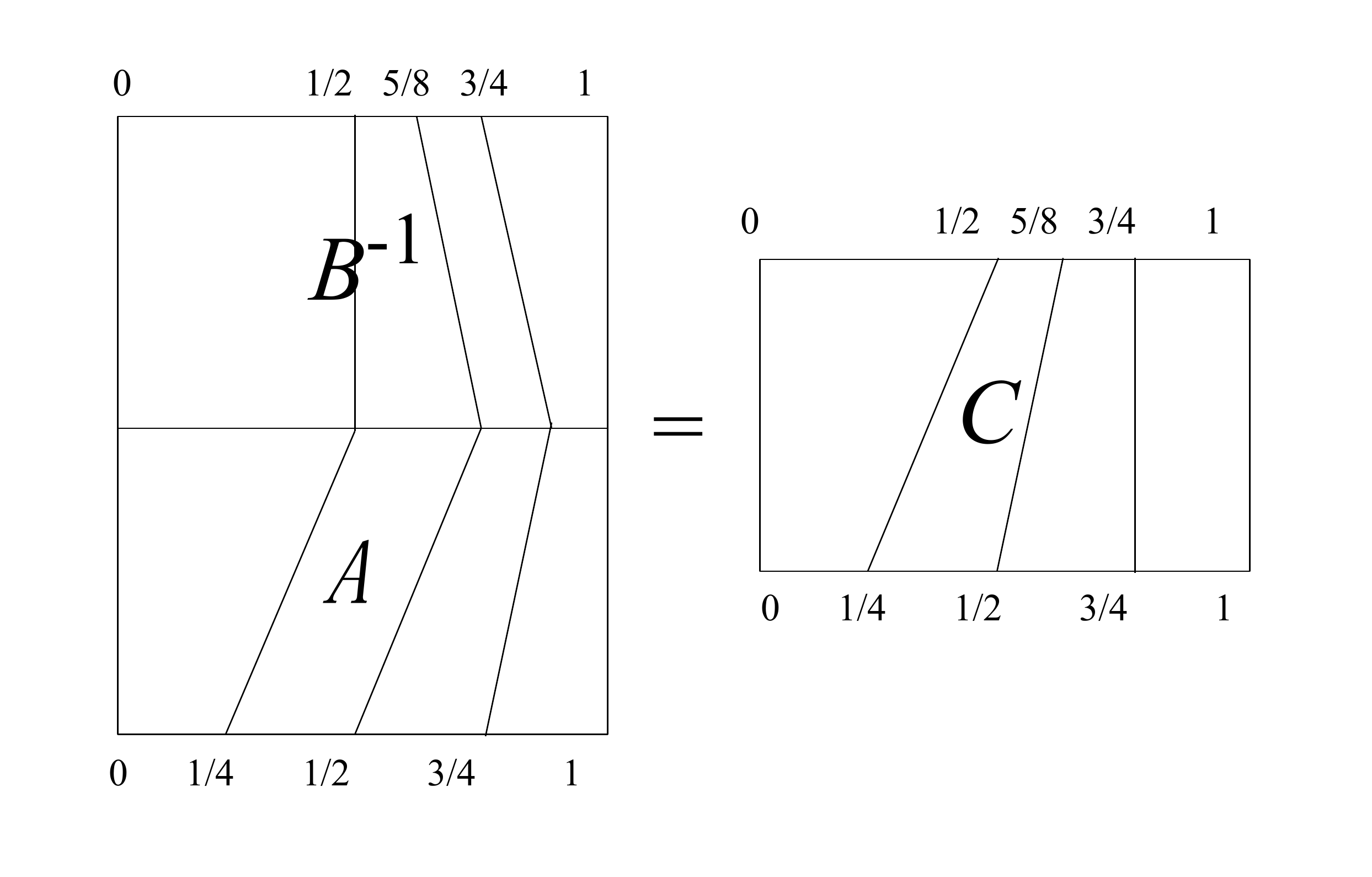}\\
  \caption{$C$ acts on $[0,1]$}\label{}
\end{figure}
\begin{figure}[H]
  \centering
  \includegraphics[width=10cm]{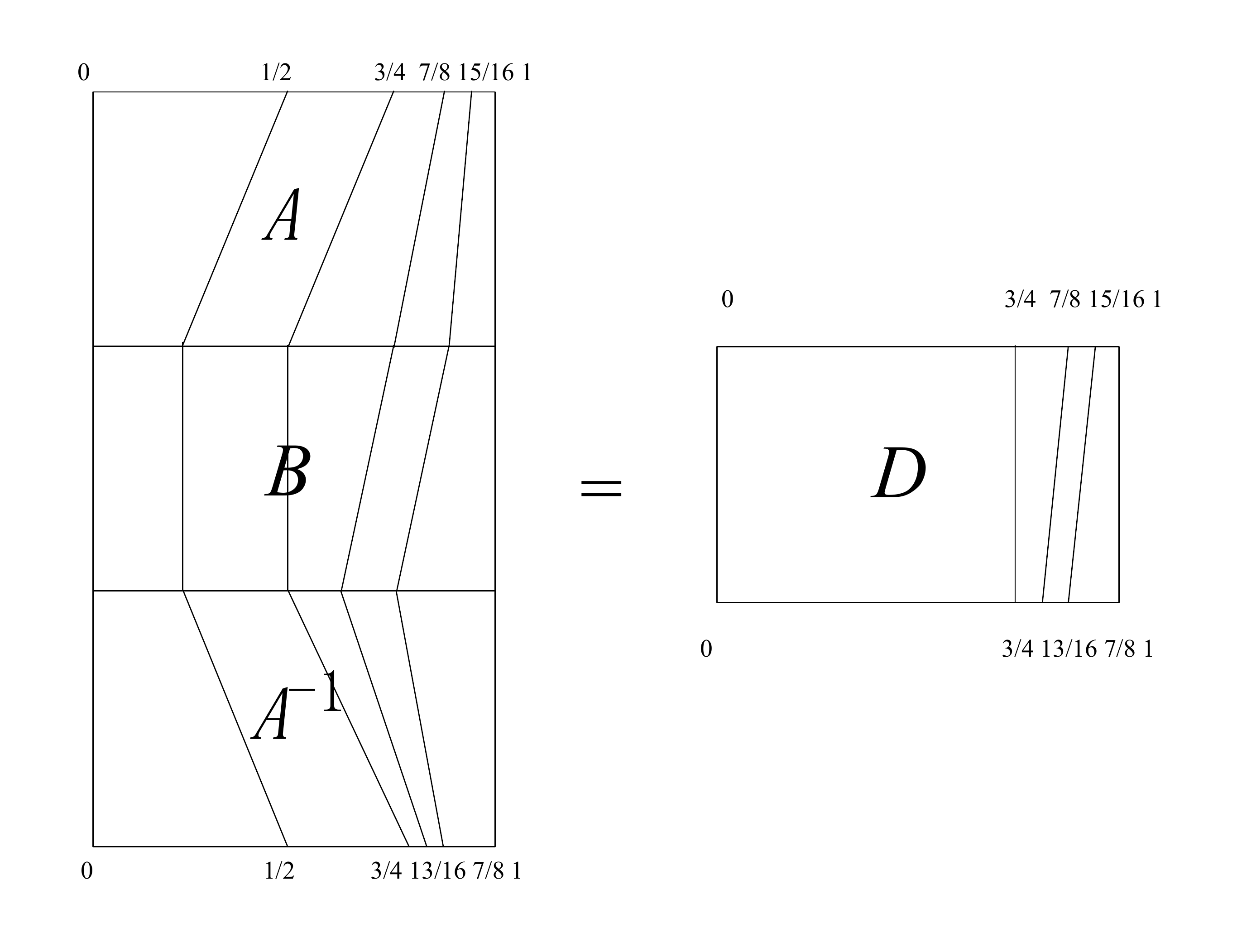}\\
  \caption{$D$ acts on $[0,1]$}\label{}
\end{figure}

Now we apply Theorem 1.1 to $H_C,H_D$ and obtain that there is 
\begin{center}
$X_C\in (B(H_C)\oplus \id_{H_D})\cap \mathcal{A}$ and $X_D\in (\id_{H_C}\oplus B(H_D))\cap\mathcal{A}$
\end{center}
which are the infinitesimal generators of the one parameter subgroups corresponding to $u_C,u_D$.
Hence $X_C,X_D$ commute. 

The proof of $[X_C,X_E]=0$ follows similarly. 
\end{proof}

\begin{corollary}
The lifting of the commutators $[C,D]$ and $[C,E]$ into $U(\mathcal{F})$ are given by
\begin{center}
$[U_C,U_D]=e^{-2i\operatorname{Im}\tr(PX_C(1-P)X_D)}$,\\
$[U_C,U_E]=e^{-2i\operatorname{Im}\tr(PX_C(1-P)X_E)}$.    
\end{center}

\end{corollary}

One can further obtain
\begin{center}
$[U_C,U_D]=e^{-2i\operatorname{Im}\tr(PX_C(1-P)X_D)}=e^{-2i\operatorname{Im}\tr(X_C^{(1,2)}X_D^{(2,1)})}$, \\
$[U_C,U_E]=e^{-2i\operatorname{Im}\tr(PX_C(1-P)X_E)}=e^{-2i\operatorname{Im}\tr(X_C^{(1,2)}X_E^{(2,1)})}$.
\end{center}
Note that as $X_C,X_D,X_E\in \mathcal{A}$, $X_C^{(1,2)},X_D^{(1,2)},X_E^{(1,2)}$ are Hilbert-Schmidt. 

\begin{corollary}
If $X_C^{(1,2)}X_D^{(2,1)},X_D^{(1,2)}X_E^{(2,1)}$ are real, we have $[U_C,U_D]=[U_C,U_E]=1$ and the projective representation $\Gamma$ can be lifted to an ordinary one. 
\end{corollary}

We can show $u_C,u_D,u_E\in B(H)$ are real in the Koopman representation (Appendix).  
Assume we can apply the following logarithm (where we have to prove the convergence at first): 
\begin{center}
$\log u=-\sum\limits_{k=1}^{\infty}\frac{1}{k}(1-u)^k$. 
\end{center}
We get $\log u_C=iX_C,\log u_D=iX_D,\log u_E=iX_E$ so that $iX_C,iX_D,iX_E$ are real and hence $X_C^{(1,2)}X_D^{(2,1)},X_C^{(1,2)}X_E^{(2,1)}$ are real. 
Then the lifting are straightforward. 

\subsection{Projections in $\text{SO}(2)$ and $\text{U}(2)$}

In Section 3.1, we define a new orthonormal basis $\{p_{n,t},q_{n,t}\}_{n\in \mathbb{N},0\leq t\leq 2^{n-2}-1}$ given by
\begin{equation*}
 \left(
  \begin{array}{c}
          p_{n,t} \\
          q_{n,t}
 \end{array}
 \right)
 =M\cdot
  \left(
  \begin{array}{c}
          f_{n,2t} \\
          f_{n,2t+1}
 \end{array}
 \right)
\end{equation*}
where
\begin{equation*}
 M=\left(
  \begin{array}{cc}
  \frac{\sqrt{2}}{2}&\frac{\sqrt{2}}{2} \\
          \frac{\sqrt{2}}{2}&-\frac{\sqrt{2}}{2}
 \end{array}
 \right)
\end{equation*}

We may let $M$ to be any matrix in $\text{SO}(2)$ or $\text{U}(2)$.
Now, assume $M\in \text{U}(2)$ and there is another orthonormal basis $B_M=\{p_1^M,q_1^M\}\cup\{p^M_{n,t},q^M_{n,t}\}_{n\geq 2,0\leq t\leq 2^{n-2}-1}$ given by
\begin{equation*}
 \left(
  \begin{array}{c}
          p_{n,t} \\
          q_{n,t}
 \end{array}
 \right)
 =M\cdot
  \left(
  \begin{array}{c}
          f_{n,2t} \\
          f_{n,2t+1}
 \end{array}
 \right)
\end{equation*}
with the $p^M_1=p_1,q^M_1=q_1$ defined in Section 3.1. 

Now, let $P_M\subset \mathcal{B}(H)$ to be the projection onto the subspace $K_M=\spn\{p_1^M,\{p^M_{n,t}\}_{n\geq 2,0\leq t\leq 2^{n-2}-1}\}$.
And there is also a Fermionic Fock space
\begin{center}
$\mathcal{F}_M=\mathcal{F}_{P_M}=\wedge(P_MH)\widehat{\otimes}\wedge(P_M^{\bot} H)^{*}$.
\end{center}
Every result in Section 4.1 follows similarly.
So that we get a projective unitary representation
\begin{center}
$\rho_M:F\to PU(\mathcal{F}_M)$.
\end{center}

According to Proposition 4.4, once the lifting of $\rho_M(A),\rho_M(B)$ are fixed as $U_M(A),U_M(B)\in U(\mathcal{F}_M)$,
it gives a value $\Psi(M)=(\alpha_M,\beta_M)\in S^1\times S^1$ by
\begin{enumerate}
\item $\alpha_M=[U_{M}(A)U_{M}(B)^{-1},U_{M}(A)^{-1}U_{M}(B)U_{M}(A)]$;
\item $\beta_M=[U_{M}(A)U_{M}(B)^{-1},U_{M}(A)^{-2}U_{M}(B)U_{M}(A)^2]$.
\end{enumerate}

Hence we obtain a well-defined map which indicates the lifting problem.  
\begin{proposition}
For any $M\in \text{U}(2)$, we get a map
\begin{center}
$\Psi:\text{U}(2) \rightarrow S^1\times S^1$ by $\Psi(M)=(\alpha_M,\beta_M)$
\end{center}
and $\rho_M$ can be lifted iff $\Psi(M)=(1,1)$. 
Moreover, when $M\in \text{SO}(2)$, we always have
\begin{center}
$\Psi:\text{SO}(2) \rightarrow \mathbb{R}^2\cap (S^1\times S^1)=\{-1,1\}\times \{-1,1\}$.
\end{center}
\end{proposition}

\appendix
\section*{Appendices}
\addcontentsline{toc}{section}{Appendices}
\renewcommand{\thesubsection}{\Alph{subsection}}

\subsection{The matrices $u_A,u_B$}

The actions of $u_A,u_B$ on lower terms ($p_{n,t},q_{n,t}$ with $n\leq 3,4$ respectively) are given by
\begin{center}
$u_A=\begin{pmatrix}
\ddots &  \vdots & \vdots & \vdots & \vdots & \vdots & \vdots & \vdots & \vdots &\udots \\
\ldots &  0 & 0 & 0 & 0 & 0 & 0 & 0& 0 &\ldots\\
\ldots  & \frac{1}{2} & \frac{1}{2} & 0 & 0 & -\frac{1}{2} & -\frac{1}{2}  & 0&0&\ldots\\
\ldots &  0 & 0 & 0 & 0 & 0 & 0 & 0&0&\ldots\\
\ldots & 0 &\frac{1}{2} & -\frac{1}{4} &  \frac{\sqrt{2}}{4} & -\frac{1}{2} & -\frac{1}{4} & \frac{1}{2} & 0&\ldots\\
\ldots &  0 &0 &\frac{1}{2} & \frac{\sqrt{2}}{2} & 0 & \frac{1}{2} & 0 &  0&\ldots\\
\ldots  & 0 &0 & -\frac{\sqrt{2}}{4} & \frac{1}{2} & \frac{\sqrt{2}}{2} & -\frac{\sqrt{2}}{4} & 0 & 0&\ldots\\
\ldots & 0 &\frac{1}{2} & \frac{1}{4} & -\frac{\sqrt{2}}{4} & \frac{1}{2} & \frac{1}{4} & \frac{1}{2} & 0&\ldots\\
\ldots  & 0 & 0 & 0 & 0 & 0 & 0 & 0 &  1&\ldots\\
\dots  & 0 &\frac{1}{2} & -\frac{1}{2} & 0 & 0 & \frac{1}{2} & -\frac{1}{2}  &0&\ldots\\
\dots  & 0 & 0 & 0 & 0 & 0 & 0 & 0& 0&\ldots\\
\udots  & \vdots & \vdots & \vdots & \vdots & \vdots & \vdots & \vdots & \vdots &\ddots 
\end{pmatrix}$
\end{center}

\begin{center}
$u_B=\begin{pmatrix}
\ddots&\vdots&\vdots&\vdots&\vdots&\vdots&\vdots&\vdots&\vdots&\vdots&\vdots&\vdots&\vdots&\vdots&\vdots&\vdots&\vdots&\udots\\
\ldots&0&\frac{1}{2}&0&0&\frac{1}{2}&0&0&0& 0&0&0&-\frac{1}{2}&0&0&-\frac{1}{2}&0&\ldots\\
\ldots&0&0&0&0&0&0&0&0&0&0&0&0&0&0&0&0& \ldots\\
\ldots&0&0&1&0&0&0&0&0&0&0&0&0&0&0&0&0& \ldots\\
\ldots&0&0&0&1&0&0&0&0&0&0&0&0&0&0&0&0& \ldots\\
\ldots&0&\frac{1}{2}&0&0&-\frac{1}{4}&0&\frac{2+\sqrt{2}}{8}&0&\frac{1-\sqrt{2}}{4}&-\frac{2+\sqrt{2}}{8}&0&-\frac{1}{4}&0&0&\frac{1}{2}&0& \ldots\\
\ldots&0&0&0&0&0&1&0&0&0&0&0&0&0&0&0&0& \ldots\\
\ldots&0&0&0&0&\frac{2+\sqrt{2}}{8}&0&\frac{3+\sqrt{2}}{8}&0&-\frac{\sqrt{2}}{8}&\frac{5-\sqrt{2}}{8}&0&\frac{2+\sqrt{2}}{8}&0&0&0&0& \ldots\\
\ldots&0&0&0&0&0&0&0&1&0&0&0&0&0&0&0&0& \ldots\\
\ldots&0&0&0&0&\frac{-1+\sqrt{2}}{4}&0&\frac{\sqrt{2}}{8}&0&\frac{1+2\sqrt{2}}{4}&\frac{-\sqrt{2}}{8}&0&\frac{-1+\sqrt{2}}{4}&0&0&0&0& \ldots\\
\ldots&0&0&0&0&-\frac{2+\sqrt{2}}{8}&0&\frac{5-2\sqrt{2}}{8}&0&\frac{\sqrt{2}}{8}&\frac{3+2\sqrt{2}}{8}&0&-\frac{2+\sqrt{2}}{8}&0&0&0&0& \ldots\\
\ldots&0&0&0&0&0&0&0&0&0&0&1&0&0&0&0&0& \ldots\\
\ldots&0&\frac{1}{2}&0&0&\frac{1}{4}&0&-\frac{2+\sqrt{2}}{8}&0&\frac{-1+\sqrt{2}}{4}&\frac{2+\sqrt{2}}{8}&0&\frac{1}{4}&0&0&\frac{1}{2}&0& \ldots\\
\ldots&0&0&0&0&0&0&0&0&0&0&0&0&1&0&0&0& \ldots\\
\ldots&0&0&0&0&0&0&0&0&0&0&1&0&0&1&0&0& \ldots\\
\ldots&0&0&0&0&0&0&0&0&0&0&0&0&0&0&0&0& \ldots\\
\ldots&0&\frac{1}{2}&0&0&-\frac{1}{2}&0&0&0& 0&0&0&\frac{1}{2}&0&0&-\frac{1}{2}&0&\ldots\\
\udots&\vdots&\vdots&\vdots&\vdots&\vdots&\vdots&\vdots&\vdots&\vdots&\vdots&\vdots&\vdots&\vdots&\vdots&\vdots&\vdots&\ddots
\end{pmatrix}$
\end{center}

The actions of $u_A,u_B$ on other terms($p_{n,t},q_{n,t}$ with $n\geq 4,5$ respectively) are determined as follows.

\begin{alignat*}{2}
u_A(p_{n,t})=
\left\{
             \begin{array}{ll}
             p_{n+1,t}, & \text{if~} \frac{t}{2^{n-2}}\in [0,\frac{1}{2}] \\
            p_{n,t-2^{n-4}}, & \text{if~} \frac{t}{2^{n-2}}\in (\frac{1}{2},\frac{3}{4}] \\p_{n-1,t-2^{n-3}}, & \text{if~} \frac{t}{2^{n-2}}\in (\frac{3}{4},1] \\
             \end{array}
\right.
\qquad 
u_A(q_{n,t})=
\left\{
             \begin{array}{ll}
             q_{n+1,t}, & \text{if~} \frac{t}{2^{n-2}}\in [0,\frac{1}{2}] \\
            q_{n,t-2^{n-4}}, & \text{if~} \frac{t}{2^{n-2}}\in (\frac{1}{2},\frac{3}{4}] \\q_{n-1,t-2^{n-3}}, & \text{if~} \frac{t}{2^{n-2}}\in (\frac{3}{4},1] \\
             \end{array}
\right.
\end{alignat*}
\begin{alignat*}{2}
u_B(p_{n,t})=
\left\{
             \begin{array}{ll}
             p_{n,t}, & \text{if~} \frac{t}{2^{n-2}}\in [0,\frac{1}{2}] \\
             p_{n+1,t+2^{n-3}}, & \text{if~} \frac{t}{2^{n-2}}\in (\frac{1}{2},\frac{3}{4}] \\
            p_{n,t-2^{n-5}}, & \text{if~} \frac{t}{2^{n-2}}\in (\frac{3}{4},\frac{7}{8}] \\p_{n-1,t-2^{n-3}}, & \text{if~} \frac{t}{2^{n-2}}\in (\frac{7}{8},1] \\
             \end{array}
\right.
\qquad 
u_B(q_{n,t})=
\left\{
             \begin{array}{ll}
             q_{n,t}, & \text{if~} \frac{t}{2^{n-2}}\in [0,\frac{1}{2}] \\
             q_{n+1,t+2^{n-3}}, & \text{if~} \frac{t}{2^{n-2}}\in (\frac{1}{2},\frac{3}{4}] \\
            q_{n,t-2^{n-5}}, & \text{if~} \frac{t}{2^{n-2}}\in (\frac{3}{4},\frac{7}{8}] \\q_{n-1,t-2^{n-3}}, & \text{if~} \frac{t}{2^{n-2}}\in (\frac{7}{8},1] \\
             \end{array}
\right.
\end{alignat*}

\begin{remark}
We can check that $u_A^{(1,2)}$ is contained in a $4\times 4$ matrix and $u_B^{(1,2)}$ is contained in a $8\times 8$ matrix. 
\end{remark}

\subsection{The infinitesimal generators $X_C,X_D,X_E$}

incomplete

\end{document}